\documentclass[a4paper]{amsart}

\usepackage{amsmath}
\usepackage{amsfonts}
\usepackage{amsthm}
\usepackage{graphicx}
\usepackage{eucal}
\usepackage{amscd}
\usepackage[all,2cell]{xy}
\usepackage{amssymb}
\usepackage{mathrsfs}

 \usepackage{tikz}
 \usetikzlibrary{positioning}
 \tikzset{mynode/.style={draw,circle,inner sep=1pt,outer sep=0pt}}

\newtheorem{teo}{Theorem}[section]
\newtheorem{cor}[teo]{Corollary}
\newtheorem{lem}[teo]{Lemma}

\newdir{ |>}{{}*!/-3.5pt/@{|}*!/-8pt/:(1,-.2)@^{>}*!/-8pt/:(1,+.2)@_{>}}

\dedicatory{}

\begin{document}

\title{A NOTE ON IDEMPOTENT SEMIRINGS}

\author{George Janelidze}
\address[George Janelidze]{Department of Mathematics and Applied Mathematics, University of Cape Town, Rondebosch 7700, South Africa}

\dedicatory{Dedicated to Themba Dube on the occasion of his 65th birthday}

\thanks{}
\email{george.janelidze@uct.ac.za}

\author{Manuela Sobral}
\address[Manuela Sobral]{CMUC and Departamento de
Matem\'atica, Universidade de Coimbra, 3001--501 Coimbra,
Portugal}
\thanks{
Partially supported by the Centre for Mathematics of
the University of Coimbra -- UID/MAT/00324/2020}
\email{sobral@mat.uc.pt}



\begin{abstract}
For a commutative semiring $S$, by an $S$-algebra we mean a commutative semiring $A$ equipped with a homomorphism $S\to A$. We show that the subvariety of $S$-algebras determined by the identities $1+2x=1$ and $x^2=x$ is closed under non-empty colimits. The (known) closedness of the category of Boolean rings and of the category of distributive lattices under non-empty colimits in the category of commutative semirings both follow from this general statement.
\end{abstract}

\date{\today}

\maketitle

\section{Introduction}

Let us recall:\vspace{1mm}

\textbf{1.1.} A \emph{commutative semiring} is an algebraic structure of the form $S=(S,0,+,1,\cdot)$ in which $(S,0,+)$ and $(S,1,\cdot)$ are commutative monoids with $$x0=0,\,\,\,x(y+z)=xy+xz$$ for all $x,y,z\in S$. Here and below we use standard notational agreements: e.g. $xy+xz$ means $(x\cdot y)+(x\cdot z)$. The category of commutative semirings will be denoted by $\mathsf{CSR}$; as expected, its morphisms are semiring homomorphisms, that is, maps $f:S\to S'$ of (commutative) semirings with $$f(0)=0,\,\,\,f(s+t)=f(s)+f(t),\,\,\,f(1)=1,\,\,\,f(st)=f(s)f(t)$$ for all $s,t\in S$.

\textbf{1.2.} For a (commutative) semiring $S$, an $S$-\emph{module} (or $S$-\emph{semimodule}) is an algebraic structure of the form $A=(A,0,+,h)$ in which $(A,0,+)$ is a commutative monoid and $h:S\times A\to A$ is a map written as $(s,a)\mapsto sa$ that has $$1a=a,\,\,\,s(ta)=(st)a,\,\,\,s0=0,\,\,\,s(a+b)=sa+sb,\,\,\,0a=0,\,\,\,(s+t)a=sa+ta$$ for all $s,t\in S$ and $a,b\in A$. The category of $S$-modules will be denoted by $S$-$\mathsf{mod}$.

\textbf{1.3.} For $S$ above, a \emph{commutative $S$-algebra} is a pair $(A,f)$ in which $A$ is a commutative semiring, and $f:S\to A$ is a semiring homomorphism. Accordingly, the category of commutative $S$-algebras is just the comma category $(S\downarrow\mathsf{CSR})$. Equivalently, a commutative $S$-algebra can be defined as an algebraic structure of the form $A=(A,0,+,h,1,\cdot)$ in which $A=(A,0,+,h)$ is an $S$-module and $A=(A,0,+,1,\cdot)$ is a commutative semiring with $s(ab)=(sa)b$ for all $s\in S$ and $a,b\in A$; we will then have $sa=f(s)a$.\vspace{1mm}

\textbf{1.4.} If $S$ is a commutative ring, then $(S\downarrow\mathsf{CSR})$ is the ordinary category of commutative (unital) $S$-algebras. In particular: if $S=\mathbb{Z}$ is the ring of integers, then $(S\downarrow\mathsf{CSR})$ is the category $\mathsf{CRings}$ of commutative rings; if $S=\{0,1\}$ with $1+1=0$, then $(S\downarrow\mathsf{CSR})$ is the category $\mathsf{CRings}_2$ of commutative rings of characteristic $2$ (=the category of commutative algebras over the two-element field). The category $\mathsf{CRings}_2$ contains the category $\mathsf{BRings}$ of Boolean rings (=commutative rings satisfying the identity $x^2=x$).

\textbf{1.5.} If $S=\{0,1\}$ with $1+1=1$, then $(S\downarrow\mathsf{CSR})$ is the category $\mathsf{AICSR}$ of additively idempotent commutative semirings (=commutative semirings satisfying the identity $2=1$, or, equivalently, the identity $2x=x$). This category contains the category $\mathsf{DLat}$ of distributive lattices.\vspace{1mm}

In this paper we present a general result (Corollary 2.3 in the next section) on $S$-semialgebras, which implies (known) closedness of the categories of Boolean rings and of distributive lattices under non-empty colimits (or, equivalently, just under binary coproducts) in the category of commutative semirings.

\section{The general case}

Let us first mention an obvious general fact:

\begin{lem}
	Let $\mathcal{V}$ be a variety of universal algebras, $\mathcal{W}$ a subvariety of $\mathcal{V}$, and $I$ the initial object of $\mathcal{W}$ (=the free algebra in $\mathcal{W}$ on the empty set). If $\mathcal{W}\approx(I\downarrow\mathcal{W})$ is coreflective in $(I\downarrow\mathcal{V})$, then $\mathcal{W}$ is closed under non-empty colimits in $\mathcal{V}$.\qed
\end{lem}

Then we take:
\begin{itemize}
	\item $\mathcal{V}=(S\downarrow\mathsf{CSR})$;
	\item $\mathcal{W}=(S\downarrow\mathsf{CSR})^*$ to be the subvariety of $(S\downarrow\mathsf{CSR})$ consisting of all $S$-algebras satisfying the identities $1+2x=1$ and $x^2=x$. This makes $I\approx S/E$, where $E$ is the smallest congruence on $S$ containing $(1+2s,1)$ and $(s^2,s)$ for each $s\in S$.
\end{itemize}
\begin{teo}
	Let $I\approx S/E$ be as above. The variety $$(S\downarrow\mathsf{CSR})^*\approx(I\downarrow(S\downarrow\mathsf{CSR})^*)$$ of commutative $S$-algebras satisfying the identities $1+2x=1$ and $x^2=x$ is a coreflective subcategory of $(I\downarrow\mathsf{CSR})\approx(I\downarrow(S\downarrow\mathsf{CSR}))$.
\end{teo}
\begin{proof}
	It suffices to show that, for each $A\in(I\downarrow\mathcal{V})$, the set $$A'=\{a\in A\mid 1+2a=1\,\,\&\,\,a^2=a\}$$ is a subalgebra of $A$, that is, to show the following:
	\begin{itemize}
		\item [(i)] $a,b\in A'\Rightarrow a+b\in A'$;
		\item [(ii)] for each $s\in S$, $a\in A'\Rightarrow sa\in A'$;
		\item [(iii)] $1\in A'$;
		\item [(iv)] $a,b\in A'\Rightarrow ab\in A'$;
	\end{itemize}
	Indeed,	(i): Suppose $a$ and $b$ are in $A'$. Then $1+2(a+b)=1+2a+2b=1+2b=1$ and $(a+b)^2=a^2+2ab+b^2=a+2ab+b=a(1+2b)+b=a+b$.
	
	(iii): $1+2\cdot1=1$ since this equality holds in $I$.
	
	(iv): Suppose $a$ and $b$ are in $A'$. Then $$1+2ab=1+2a+2ab=1+a+a(1+2b)=1+a+a=1+2a=1$$ and $(ab)^2=a^2b^2=ab$.
	
	(ii) follows from (iv) since $sa=(s1)a$ and $(s1)$ is in $A'$ (since $s1$ is the image of class of $s$ under the homomorphism $I\to A$).
\end{proof}
From Lemma 3.1 and Theorem 3.2 we obtain:
\begin{cor}
	The variety $(S\downarrow\mathsf{CSR})^*$ of commutative $S$-algebras satisfying the identities $1+2x=1$ and $x^2=x$ is closed under non-empty colimits in the variety $(S\downarrow\mathsf{CSR})$ of all commutative $S$-algebras.\qed
\end{cor}

Taking $S$ to be the ring of natural numbers, we obtain the following special cases of Theorem 2.2 and Corollary 2.3:

\begin{cor}
	The variety $\mathsf{CSR}^*$ of commutative semirings satisfying the identities $1+2x=1$ and $x^2=x$ is coreflective in the variety $(\{0,1,2\}\downarrow\mathsf{CSR})$, where $1+2=1$ in $\{0,1,2\}$.\qed
\end{cor}
\begin{cor}
	The variety $\mathsf{CSR}^*$ above is closed under non-empty colimits in $\mathsf{CSR}$.\qed
\end{cor}

\section{Boolean rings and distributive lattices}

If an object $A$ of $(\{0,1,2\}\downarrow\mathsf{CSR})$ belongs to $(\{0,1\}\downarrow\mathsf{CSR})$ with $1+1=0$ in $\{0,1\}$ making $(\{0,1\}\downarrow\mathsf{CSR})=\mathsf{CRings}_2$, then $$\{a\in A\mid 1+2a=1\,\,\&\,\,a^2=a\}=\{a\in A\mid 2a=0\,\,\&\,\,a^2=a\}.$$ But if it is the case with $1+1=1$ making $(\{0,1\}\downarrow\mathsf{CSR})=\mathsf{AICSR}$, then $$\{a\in A\mid 1+2a=1\,\,\&\,\,a^2=a\}=\{a\in A\mid 1+a=1\,\,\&\,\,a^2=a\}.$$ Therefore we obtain the commutative diagram
$$\xymatrix{\mathsf{CRings}_2\ar[d]\ar[r]&(\{0,1,2\}\downarrow\mathsf{CSR})\ar[d]&\mathsf{AICSR}\ar[l]\ar[d]\\\mathsf{BRings}\ar[r]&\mathsf{CSR}^*&\mathsf{DLat}\ar[l]}$$ where the horizontal arrows are inclusion functors while the left-hand and right-hand vertical arrows are the coreflections induced by the coreflection of Corollary 2.4 represented by the middle vertical arrow. Since $\mathsf{CRings}_2$ and $\mathsf{AICSR}$ both being of the form $(\{0,1\}\downarrow\mathsf{CSR})$ (with different $1+1$ in $\{0,1\}$) are closed in $\mathsf{CSR}$ under non-empty colimits, we conclude that both $\mathsf{BRings}$ and $\mathsf{DLat}$ are also closed in $\mathsf{CSR}$ under non-empty colimits. That is, as promised in our Introduction, these two known results follow from what we have done in general (in Section 2).

\section{Two additional remarks}

\textbf{4.1.} The Reader might ask, what is special about $(S\downarrow\mathsf{CSR})$? The answer consists of the following observations:
\begin{itemize}
	\item $(S\downarrow\mathsf{CSR})$ is the category of commutative monoids in the monoidal category $S$-$\mathsf{mod}$ having therefore `good' colimits; indeed, its binary coproducts are given by tensor products.
	\item The monoidal category structure of $S$-$\mathsf{mod}$ is determined by the fact that it is a commutative variety of universal algebras.
	\item A commutative variety of universal algebras is semi-additive if and only if it is of the form $S$-$\mathsf{mod}$ for some commutative semiring $S$. This immediately follows from the equivalence $1.\!\Leftrightarrow5.$ in Theorem 2.1 of \cite{[JM1970]}, which refers to \cite{[C1963]} for the proof.
\end{itemize}

\textbf{4.2.} The coreflectivity of $\mathsf{DLat}$ in $\mathsf{AICSR}$ is a `finitary copy' of the coreflectivity of the category of frames in the category of quantales, see Section C1.1 of \cite{[J2002]}: in fact $A_f$ on Page 479 there is the same as our $\{a\in A\mid 1+a=1\,\,\&\,\,a^2=a\}$.

\end{document}